\documentclass[a4paper,notitlepage, 8pt,reqno]{article}

\usepackage[cp1251]{inputenc}
 \usepackage[T2A]{fontenc}
 \usepackage[russian,english]{babel}
\usepackage[all,arc,poly,2cell,curve,arrow,tips]{xy}
\usepackage{enumerate, eucal, amsthm,amsmath, amssymb}
\usepackage{graphicx}
\usepackage{mathtext}
\usepackage{amsmath,tikz}

\usetikzlibrary{tikzmark}

\usepackage{tikz}

\allowdisplaybreaks[4]

\theoremstyle{definition}
\newtheorem{defn}{Definition}

\theoremstyle{plain}
\newtheorem{thm}{Theorem}
\newtheorem*{thm*}{Теорема}
\newtheorem{prop}{Proposition}
\newtheorem*{prop*}{Предложение}

\renewcommand{\abstractname}{}

\title{ $3j$-symbols for representation of the Lie algebra  $\mathfrak{gl}_3$ in the Gelfand-Tselin base }

\author{D.V. Artamonov}
\date{}

\begin{document}
    \maketitle

 \maketitle

\renewcommand{\abstractname}{}

\begin{abstract}
In the paper a simple explicit formula for an arbitrary  $3j$-symbol for the Lie algebra  $\mathfrak{gl}_3$ is given. It is expressed through a fraction of values of   hypergeometric functions when one substitutes $\pm 1$ instead of all it's arguments.  The problem of calculation of an arbitrary  $3j$-symbol is equivalent to the problem of calculation of an arbitrary Clebsh-Gordan coefficient  for the algebra  $\mathfrak{gl}_3$.  These coefficients play an important role in quantum mechanics in the  theory of quarks.
\end{abstract}

\section{Introduction}

Consider a tensor product of irreducible representation  $V$  and $W$ of the algebra $\mathfrak{gl}_3$  and let us split it into a sum of irreducibles:
 \begin{equation}
 \label{rzl}
   V\otimes W=\sum_{U,s} U^s,    \end{equation}

where    $U$ denotes possible types of irreducible representations that occur in this decomposition and the symbol  $s$  is indexing irreducible representations   $U^s$  of type $U$, occurring in the decomposition \footnote{ A precise definition of the index   $s$ is the following. We are writing a decomposition as follows:  $V\otimes W=\sum_{U} M_U\otimes U$,  where  $M_U$  is a linear space named the multiplicity space. Let   $\{e_i\}$ be it's base,  then  $U^s:=e_s\otimes U$. }.

 Let us choose  in these representations bases $\{v_{\mu}\}$, $\{w_{\nu}\}$, $\{u^s_{\rho}\}$.  The Clebsh-Gordon coefficients are  numeric coefficients $C^{U,\rho,s}_{V,W;\mu,\nu}\in\mathbb{C}$,  appearing in the decomposition
  
  \begin{equation}
\label{kg1}
  v_{\mu}\otimes w_{\nu}=\sum_{s,\rho} C^{U,\rho,s}_{V,W;\mu,\nu} u_{\rho}^s.
\end{equation}

Also we use the term the Clebsh-Gordan coefficients for coefficients   $D^{U,\rho,s}_{V,W;\mu,\nu}\in\mathbb{C}$, occuring in the decomposition
  
  \begin{equation}
\label{kg2}
  u_{\rho}^s = \sum_{\mu,\nu} D^{U,\rho,s}_{V,W;\mu,\nu} v_{\mu}\otimes w_{\nu}.
\end{equation}


These coefficients in the case of algebras   $\mathfrak{gl}_2$, $\mathfrak{gl}_3$ play an important role in the quantum mechanics.  The Clebsh-Gordan coefficients for the algebra   $\mathfrak{gl}_2$ are used in the spin theory   (see \cite{blb}),  and  the Clebsh-Gordan coefficients for the algebra $\mathfrak{gl}_3$ are used in the theory of quarks (see \cite{GrM}).  The problem of their calculation in the case   $\mathfrak{gl}_2$ is quite simple. There are explicit formulas that were first obtained by Van der Varden  \cite{Gkl0}.  The Clebsh-Gordan coefficients for the algebra  $\mathfrak{sl}_2$ allow to construct new realization of representations of real Lie groups associated with this Lie algebra ($SU(2)$ in \cite{blb}, \cite{Go}; \cite{GoGo}; $SO(3)$  in  \cite{GM}, \cite{GoGo}, \cite{Go}).  They appear in the elasticity theory (see \cite{Sel}, \cite{Sel2}, where actually the  $SO(3)$ group is considered).

In the application it is especially important to find the  Clebsh-Gordan coefficients in the case when in representations the Gelfand-Tsetlin base is taken.

The problem of calculation of  Clebsh-Gordan coefficients for  $\mathfrak{gl}_n$  when  $n\geq 3$  is much more difficult than in the case    $n=2$.  In the case  $n=3$  they for the first time were calculated in a series of papers by Biedenharn, Louck, Baird  \cite{bb1963}, \cite{bl1968}, \cite{bl1970}, \cite{bl19731}, \cite{bl19732}. In these papers the general case   $\mathfrak{gl}_n$ is considered,  but only in the case $n=3$  their calculations allow to obtain in principle a formula for a general Clebsh-Gordan coefficient.  The calculations are based on the following principle. Consider tensor opertors between two  representations of  $\mathfrak{gl}_n$, i.e. collections  $\{f_u\}$ of mappings

$$
f_{u}: V \rightarrow  W
$$

between representations of   $\mathfrak{gl}_n$,   the mappings   $\{f_u\}$  are in one-to-one correspondence with  vectors of a representation   $U$ of  $\mathfrak{gl}_n$.  Some condition  relating the actions of  $\mathfrak{gl}_n$ on these representations must hold.  The matrix elements of tensor operators are closely related to the Clebsh-Gordan coefficients  (The Wigner Eckart theorem ).  In papers  \cite{bb1963}-\cite{bl19732}  some explicit realization of these operators is given (in different papers cases of different representations $U$ are considered).   Such an explicit realization does not allow to obtain exlicit formulas for  matrix elements of a tensor operator corresponding tо a given $U$  but it allows to express them through the matrix elelments of tensor operators corresponding to   $U$,  considered in previous papers.

Unfortunately there is no explicit formula in their papers. It is just clear that it can be obtained.

Thus the problem of finding of an {\it explicit and simple
	 } formula for a general Clebsh-Gordan coefficient for the algebra $\mathfrak{gl}_3$ still remained unsolved.  A review of proceeding papers can be found in   \cite{a1}.
 In this paper for the first time a really  {\it explicit } formula for a general Clebsh-Gordan coefficient for   $\mathfrak{gl}_3$  in the decomposition \eqref{kg2} was obtained. That is an explicit formula of type   $D^{U,\gamma,s}_{V,W;\alpha,\beta}=...$ was derived. Unfortunately this formula is quite cumnbersome, thus the problem of derivation of a  {\it simple } formul remained unsolved.

For the purpose of caculation of Clebsh-Gordan coefficients it is necessary to choose an explicit realization of a representation of $\mathfrak{gl}_3$.
A realization which is very convenient  for calculation was suggested in   \cite{bb1963}.  In this paper the following is proved. If one uses a realization of a representation in the space of functions on the Lie group   $GL_3$, then functions corresponding to Gelfand-Tsetlin base vectors can be expressed through the Gauss' hypergeometric function (see a modern viewpoint in   \cite{a2}).  This idea was used in \cite{a1}.

In the present paper we change an approach of  \cite{a1} and this allows to obtain a much more simpler result. Instead of coefficients in the decomposition \eqref{kg1} we are calculating the  $3j$-symbols (see their definition and their relation to Clebsh-Gordan coefficients in Section  \ref{3jcg}). We 
express the  $3j$-symbols through the values of an hypergeometric function.   As the result we manage to obtain  {\it explicit and simple}  formulas for a general Clebsh-Gordan coefficient for the algebras  $\mathfrak{gl}_3$.   The main result is formulated in  Theorem  \ref{ost},  where 
  a formula for a   $3j$-in this case    (see  \eqref{osntf2}).

In Appendix  \ref{dop} we give s selection rulers for Clebsh-Gordan coefficients and   $3j$-symbols.

\section{ The basic notions}

\subsection{  $A$-hypergeometric functions}

One can find information about a  $\Gamma$-series in \cite{GG}.

Let $B\subset \mathbb{Z}^N$  be a lattice, let $\mu\in \mathbb{Z}^N$ be a fixed vector. Define a  {\it  hypergeometric
	$\Gamma$-series }  in variables  $z_1,...,z_N$ by formulas

\begin{equation}
\label{gmr}
\mathcal{F}_{\mu}(z,B)=\sum_{b\in
	B}\frac{z^{b+\mu}}{\Gamma(b+\mu+1)},
\end{equation}
where $z=(z_1,...,z_N)$, and we use the multiindex notations

$$
z^{b+\mu}:=\prod_{i=1}^N
z_i^{b_i+\mu_i},\,\,\,\Gamma(b+\mu+1):=\prod_{i=1}^N\Gamma(b_i+\mu_i+1).
$$

Note that in the case when at least one of the components of the vector  $b+\mu$ is non-positive integer then the corresponding summand in  \eqref{gmr} vanishes. Thus in the considered below   $\Gamma$-series there are only finitely many terms. Also we write below factorials instead of $\Gamma$-functions.

A $\Gamma$-series satisfies the Gelfand-Kapranov-Zelevinsky system. Let us write it in the case $z=(z_1,z_2,z_3,z_4)$, $B=\mathbb{Z}<(1,-1,-1,1)>$:
\begin{align}
\begin{split}
\label{gkzs} \Big(\frac{\partial^2}{\partial z_1\partial
	z_4}-\frac{\partial^2}{\partial z_2\partial z_3}\Big)F_{\mu,B}&=0,
\\
z_1\frac{\partial}{\partial z_1}F_{\mu,B}+ z_2\frac{\partial}{\partial
	z_2}F_{\mu,B}& =(\mu_1+\mu_2)F_{\mu,B},\quad
\\
z_1\frac{\partial}{\partial z_1}F_{\mu,B}+
z_3\frac{\partial}{\partial z_3}F_{\mu,B}&=(\mu_1+\mu_3)F_{\mu,B},
\\
z_1\frac{\partial}{\partial z_1}F_{\mu,B}-z_4\frac{\partial}{\partial
	z_4}F_{\mu,B}&=(\mu_1-\mu_4)F_{\mu,B}.
\end{split}
\end{align}

\subsection{ A functional realization of a Gelfand-Tsetlin base}

In the paper Lie algebras and groups over   $\mathbb{C}$ are considered.

Functions on $GL_3$ form a representation of the group  $GL_3$. On a function  $f(g)$, $g\in GL_3$, an element  $X\in GL_{3}$  acts by right shifts

\begin{equation}
\label{xf} (Xf)(g)=f(gX).
\end{equation}

Passing to an infinitesimal action we obtain that on the space  of all functions on  $GL_3$ there exists an action of  $\mathfrak{gl}_3$.

Every finite dimensional irreducible representation can be realized as a subrepresentation in the space of functions.
Let
$[m_{1},m_2,m_{3}]$ be a highest weight, then in the space of functions there is a  highest vector with such a weight, which is written explicitly as follows.

Let $a_{i}^{j}$, $i,j=1,2,3$ -
be a function of a matrix element on the group  $GL_{3}$. Here $j$  is a row index and $i$  is a column index.
Also put

\begin{equation}
\label{dete}
a_{i_1,...,i_k}:=det(a_i^j)_{i=i_1,...,i_k}^{j=1,...,k},
\end{equation}

 where we take a determinant of a submatrix in a matrix  $(a_i^j)$,
formed by rows with indices  $1,...,k$  and columns with indices $i_1,...,i_k$.
 The operator $E_{i,j}$ acts onto determinants by transforming their column indices

\begin{equation}
\label{edet1}
E_{i,j}a_{i_1,...,i_k}=a_{\{i_1,...,i_k\}\mid_{j\mapsto i}},
\end{equation}

where  $.\mid_{j\mapsto i}$ denotes an operation of substitution $i$ instead of   $j$
$i$;  if  $j$  does not occur among $\{i_1,...,i_k\}$,  then we obtain zero.

One sees that a rising operator     $E_{i,j}$, $i<j$ tryes to change an index to a smaller one. Thus the function 
\begin{equation}
\label{stv}
v_0=\frac{a_{1}^{m_{1}-m_{2}}}{(m_1-m_2)!}\frac{a_{1,2}^{m_{2}-m_{3}}}{(m_2-m_3)!}\frac{a_{1,2,3}^{m_{3}}}{m_{3}!}
\end{equation}

is a highest vector  for the algebra  $\mathfrak{gl}_{3}$ with the weight
$[m_{1},m_{2},m_3]$.


Let us write a formula for a function corresponding  to a Gelfand-Tselin diagram for  $\mathfrak{gl}_3$ (see definition of this base in   \cite{zh}). A diagram is an integer table of the following type in which the betweeness conditions hold

\begin{align*}
\begin{pmatrix}
m_{1} && m_{2} &&0\\ &k_{1}&& k_{2}\\&&s
\end{pmatrix}
\end{align*}
A formula for a function corresponding to this diagram is given in the Theorem proved in \cite{bb1963}.

\begin{thm}\label{vec3}  Put $B_{GC}=\mathbb{Z}<(0,1,-1,-1,1,0)>$,   $\mu=(m_1-k_1,s-m_{2},k_{1}-s_{},m_{2}-k_{2},0,k_2)$, then to a diagram there corresponds  a function $$
\mathcal{F}_{\mu}(a_3,a_{1},a_{2},a_{1,3},a_{2,3},a_{1,2},B_{GC})
$$

\end{thm}

In \cite{bb1963} an expression involving the Gauss' hypergeometric function is given. A modern form of this formula involving  a   $\Gamma$-series was presented in  \cite{a2}.

To make notations shorter in the case when we use a   $\Gamma$-series corresponding to a lattice  $B_{BG}$  from Theorem  \ref{vec3} we  omit the notation $B_{GC} $ and we write just  $\mathcal{F}_{\mu}(a)$.  In the case when we use another lattice we do not omit it.

Note also that the first equation form the GKZ system looks as follows

\begin{equation}
\label{gkz}
\mathcal{O}\mathcal{F}=0,\,\,\,\mathcal{O}=\frac{\partial^2}{\partial a_1\partial a_{2,3}}-\frac{\partial^2}{\partial a_2\partial a_{1,3}}.
\end{equation}

\subsection{ A-GKZ system and it's solutions. A-GKZ realization }

\subsubsection{ A-GKZ system. A  base  $F_{\mu}$    and a base   $\tilde{F}_{\mu}$ in it's  solution space}

Instead of determinants   $a_X$, $X\subset \{1,2,3\}$,  that satisfy the Plucker relations let us introduce variables  $A_X$, $X\subset \{1,2,3\}$,  we suppose that  $A_X$ are skew-symmetric functions of   $X$.

Let us return to the GKZ-system   \eqref{gkz}   and let us change the differential operators that define it. 
Consider  functions $F(A)$, that satisfies the equations 

\begin{equation}
\label{agkz}
\bar{\mathcal{O}}_{A}F=0,   \,\,\, \bar{\mathcal{O}}_{A}=\frac{\partial^2}{\partial A_1\partial A_{2,3}}-\frac{\partial^2}{\partial A_2\partial A_{1,3}}+\frac{\partial^3}{\partial A_3\partial A_{1,2}}
\end{equation}

This system is called an antisymmetrized Gelfand-Kapranov-Zelevinsky system(or A-GKZ for short).

Let us find a base in the space of polynomial solutions of such a system.

Using the equality  (65) from  \cite{a1} one can show that the following function is a solution 

\begin{equation}
\label{Fmu}
F_{\mu}(A):=\sum_{s\in\mathbb{Z}_{\geq 0}}  q^{\mu}_s \zeta_A^{s}\mathcal{F}_{\mu-s(e_3+e_{1,2})}(A), 
\end{equation}

where  \begin{align}\begin{split}\label{cs} &t^{\mu}_0=1, \,\,\,\,\, t^{\mu}_s=\frac{1}{s(s+1)+s(\mu_1+\mu_2+\mu_{1,3}+\mu_{2,3})}, \text{ при }s>0\\& q_{\mu}^s=\frac{t_s^{\mu}}{\sum_{s'\in\mathbb{Z}_{\geq 0}}  t_{s'}^{\mu}},  \,\,\,\,\,
\zeta_A=A_1A_{2,3}-A_{2}A_{1,3}.
\end{split}\end{align}

One has  $F_{\mu}=\mathcal{F}_{\mu}\cdot  const$ modulo the Plucker relations.

In the space of solutions of the system   \eqref{agkz} one can construct another base.  Put
$$
v=e_1+e_{2,3}-e_{2}-e_{1,3},\,\,\,\, r=e_3+e_{1,2}-e_1-e_{2,3},
$$
and for all $s\in\mathbb{Z}_{\geq 0}$ consider a function

\begin{equation}
\mathcal{F}^s_{\mu}(A):=\sum_{t\in \mathbb{Z}}\frac{(t+1)...(t+s-1)A^{\mu+tv}}{\Gamma(\mu+tv+1)}
\end{equation}

One proves directely that 

\begin{equation}
\label{Ftildmu}
\tilde{F}_{\mu}(A)=\sum_{s\in\mathbb{Z}_{\geq 0}} \frac{(-1)^{s}}{s!} \mathcal{F}^s_{\mu-sr}(A)
\end{equation}

is also a solution of   \eqref{agkz}.   Let us introduce  an order on shift vectors considered   $mod B_{GC}$
\begin{equation}
\label{por}
\mu\preceq \nu \Leftrightarrow  \mu=\nu-sr \,\, mod B_{GC}\,\,\, s\in\mathbb{Z}_{\geq 0}.
\end{equation}

Then considering supports of functions   $F_{\mu}(A)$, $\tilde{F}_{\mu}(A)$ one comes to a conclusion that the base   $\tilde{F}_{\mu}$ is related to the collection of functions  $F_{\mu}$ by a   low-unitriangular relatively the ordering    \eqref{por} linear transformation. That is

$$
\tilde{F}_{\mu}=\sum_{s\in \mathbb{Z}_{\geq 0}}d_s F_{\mu-sr},\,\,\, d_0=1
$$

Hence functions   $\tilde{F}_{\mu}(A)$ form a base in the solution space of the system   \eqref{agkz}.

\subsubsection{ A realization of a representation in the space of solutions of the A-GKZ system}
Define an action onto the variables  $A_X$ of the algebra    $\mathfrak{gl}_3$  by the ruler:

$$
E_{i,j}A_X=\begin{cases}A_{X\mid_{j\mapsto i}}, \text{ если }j\in X, \,\,\,\text{см.} \eqref{edet1} \\0\text{ otherwise. }\end{cases}
$$

One easily checks that this is an action of the Lie algebra.  One continues this action to the polynomials in $A_X$ by the Leibnitz ruler. 


The operator  $ \bar{\mathcal{O}}_{A}$ commutes with the action of   $\mathfrak{gl}_3$.  Hence the solution space of the system \eqref{agkz} is a representation of $\mathfrak{gl}_3$.

When one applies the Plucker relations  (i.e. changes   $A_X\mapsto a_X$) the considered realization is transformed to the functional relization.
Thus the functions  $F_{\mu}$ for shift vectors $\mu$, corresponding to all possible Gelfand-Tsetlin diagrams of a irreducible representation with the highest weight  $[m_1,m_2,0]$ (see Theorem \ref{vec3}) form a representation with this highest weight.

The obtained realization is called the A-GKZ realization.

One easily checks that if one substracts from   $\mu$ the vector $r$ then a diagram transforms as follows  $k_1\mapsto k_1-1$, $k_2\mapsto k_2+1$.  Thus  $\tilde{F}_{\mu}$ also form a base in the A-GKZ realization.

\subsubsection{ An explicit form of an invariant scalar product in the A-GKZ realization}

In the A-GKZ realization one can write explicitely an invariant scalar product. Between two monomials the scalar product is define as follows:

\begin{equation}
<A_{X_1}^{\alpha_1}...A_{X_n}^{\alpha_n},A_{Y_1}^{\beta_1}...A_{Y_m}^{\beta_m}>
\end{equation}
is nonzero if and only if   $n=m$, and (maybe a permutation is needed)  $X_1=Y_1$ and $\alpha_1=\beta_1$,..., $X_n=Y_n$ and $\alpha_n=\beta_n$. In this case

\begin{equation}
<A_{X_1}^{\alpha_1}...A_{X_n}^{\alpha_n},A_{X_1}^{\alpha_1}...A_{X_n}^{\alpha_n}>=\alpha_1!\cdot...\cdot\alpha_n!
\end{equation}

One easyly proves that this product is invariant.

A function in variables  $A_X$, $X\subset \{1,2,3\}$  we denote just as $f(A)$.  Then the scalar product can be rewritten using the multi-index notations as follows.
Define an action

\begin{equation}
\label{dve}
f(A)\curvearrowright  h(A):=f(\frac{d}{dA})h(A),
\end{equation}
then

\begin{equation}
\label{skd}
<f(A),h(A)>= f(A)\curvearrowright  h(A)\mid_{A=0}.
\end{equation}

Due to the symmetry of the scalar product one can write  $ <f(A),h(A)>= h(A)\curvearrowright  f(A)\mid_{A=0}.$

\subsubsection{ A relation between a base  $F_{\mu}$  and a base  $\tilde{F}_{\mu}$  of the A-GKZ realization}

Using the constructed scalar product one can find a relation between two basis of the A-GKZ realization. 

Note that  $\mathcal{F}_{\mu}=const F_{\mu}+pl$,   where  $pl=0$ modulo Plucker relation.  Then for every function  $g(A)$, which is a solution of the A-GKZ system one has

$$
<h,\mathcal{F}_{\mu}>=const<h, F_{\mu}>.
$$

Thus to prove that 

$$
\tilde{F}_{\mu}=\sum_{s\in \mathbb{Z}_{\geq 0}}d_s F_{\mu-sr}
$$

it is sufficient to show that

\begin{equation}
\label{ur}
<\tilde{F}_{\mu},\mathcal{F}_{\nu}>=\sum_{s\in \mathbb{Z}_{\geq 0}}d_s <F_{\mu-sr},\mathcal{F}_{\nu}>
\end{equation}

Using the formula \eqref{dve} and definitions of  $\tilde{F}_{\mu}$, $\mathcal{F}_{\nu}$,  one gets that the scalar product of these function is nonzero if and only if   $\nu \preceq \mu$, that is $\nu=\mu-sr\,\,\, mod B_{GC}$. Under this condition one has

$$
<\tilde{F}_{\mu},\mathcal{F}_{\mu-sr}>=\frac{(-1)^{s}}{s!}\mathcal{F}_{\mu-sr}^s(1),
$$
where on the right hand side one writes a result of substitution of  $1$ instead of all arguments of $\mathcal{F}_{\mu-sr}^s(A)$.

Now let us find a scalar product  $<F_{\mu},\mathcal{F}_{\nu}>$.   Note that   $<\zeta^k h(A),\mathcal{F}_{\nu} >$,  where $k>0$, equal to $0$ (due to the formula  \eqref{dve} and the fact that $\zeta$ acts as a GKZ operator  wich send to $0$ the function  $\mathcal{F}_{\nu}$).  Thus the scalar product only with the first summand in  \eqref{Fmu} is non-zero, thus 

$$
<F_{\mu},\mathcal{F}_{\nu}>=<\mathcal{F}_{\mu},\mathcal{F}_{\nu}>
$$

Using the formula  \eqref{dve},  one obtaines that this expression is non-zero only if $\mu=\nu mod B_{GC}$, in this case it equals   $\mathcal{F}_{\mu}(1)$.

Thus  \eqref{ur} gives that

\begin{equation}
\label{ds}
\frac{(-1)^{s+1}}{s!}\mathcal{F}_{\mu-sr}^s(1)=d_s\mathcal{F}_{\mu-sr}(1) \Rightarrow d_s=\frac{(-1)^{s}   \mathcal{F}_{\mu-sr}^s(1) }{\mathcal{F}_{\mu-sr}(1) }
\end{equation}

We need also an invertion of that expression.

\begin{equation}
\label{fs0}
F_{\mu}=\sum_{s\in \mathbb{Z}_{s\geq 0}}f_s\tilde{F}_{\mu-sr}
\end{equation}

Thus we nedd to find an inverse matrix to the mentioned low-unitriangular matrix. One has

$$
\begin{pmatrix}
1&0&0...\\
\frac{ - \mathcal{F}_{\mu-r}^1(1) }{\mathcal{F}_{\mu-r}(1) } &1 &0...\\
\frac{  \mathcal{F}_{\mu-2r}^2(2) }{\mathcal{F}_{\mu-2r}(1) } &... &1\\
...\\
\end{pmatrix}^{-1}=\begin{pmatrix}
1&0&0...\\
\frac{  \mathcal{F}_{\mu-r}^1(1) }{\mathcal{F}_{\mu-r}(1) } &1 &0...\\
\frac{ - \mathcal{F}_{\mu-2r}^2(2) }{\mathcal{F}_{\mu-2r}(1) } &... &1\\
...\\
\end{pmatrix}
$$

Thus in \eqref{fs0} one has

\begin{equation}
\label{kfs}
f^{\mu}_s=\frac{(-1)^{s+1}   \mathcal{F}_{\mu-sr}^s(1) }{\mathcal{F}_{\mu-sr}(1) }
\end{equation}

\section{ A solution of the multiplicity problem for the Clebsh-Gordan coefficients }
\label{krtn}

In the case of the algebra  $\mathfrak{gl}_2$ different representation $U^s$ occurring in the decomposition \eqref{rzl} have different highest weights.  Thus one can use the highest weight as the index   $s$.  In the case $\mathfrak{gl}_3$ the situation is much more difficult - there appears the multiplicity problem: in the decomposition  \eqref{rzl}  a representation   $U$ of a given highest weight can occur with some multiplicity.

In the paper \cite{a1} the following solution of the problem of an explicit description of representations  $U^s$ occurring in \eqref{rzl}. One realizes  $V\otimes W$ in the space of functions on a product of groups $GL_3\times GL_3$.  The functions of a matrix element on the first factor are denoted as  $a_i^j$,  and on the second as   $b_i^j$.  Introduce functions on   $GL_3\times GL_3$:

\begin{align}
\begin{split}
\label{aab}
&(ab)_{i_1,i_2}:=det\begin{pmatrix} a_i^1\\ b_i^1
 \end{pmatrix}_{i=i_1,i_2},\,\,\,\, (aabb)_{i_1,i_2,i_3,i_4}:=a_{i_1,i_2}b_{i_3,i_4}-a_{i_3,i_4}b_{i_1,i_2},\\
 &(aab)=det\begin{pmatrix} a_i^1\\ a_i^2\\b_i^1
 \end{pmatrix}_{i=1,2,3}, (abb)=det\begin{pmatrix} a_i^1\\ b_i^1\\b_i^2
 \end{pmatrix}_{i=1,2,3}
\end{split}\end{align}

Consider a tensor product   $V\otimes W$ of representation with highest weights   $[m_1,m_2,0]$ and  $[m'_1,m'_2,0]$\footnote{Below we put  $m_3=0$, $m'_3=0$}. Then a base in the space of  $\mathfrak{gl}_3$-highest vectors is formed by the following functions. Put

\begin{equation}
\label{foo}
f(\omega,\varphi,\psi,\theta):=a_1^{\alpha}b_1^{\beta}a_{1,2}^{\gamma}b_{1,2}^{\delta}(ab)_{1,2}^{\omega}(abb)^{\varphi}(aab)^{\psi}(aabb)_{1,2,1,3}^{\theta},
\end{equation}

where

\begin{align}
\begin{split}
\label{usl0}
&\alpha+\omega+\varphi=m_1-m_2,\,\,\,\gamma+\theta+\psi=m_2,\\
&\beta+\omega+\psi=m'_1-m'_2,\,\,\,\delta+\varphi+\theta=m'_2.
\end{split}
\end{align}

The function  \eqref{foo}  is indexed not by all exponents. The reason in that the exponents $\alpha,\beta,\gamma,\delta$ can be obtained from  \eqref{usl0}.

\begin{prop}
	\label{mlt}
	In the space of   $\mathfrak{gl}_3$-highest vector there is a base consisting of functions of type   $f(0,\varphi,\psi,\theta)$ and $f(\omega,\varphi,\psi,0)$.
\end{prop}

Thus an index   $s$ from  \eqref{rzl} runs through the set of functions  $f(0,\varphi,\psi,\theta)$ and  $f(\omega,\varphi,\psi,0)$, where  $f$ is defined in  \eqref{foo},  and the exponents satisfy conditions \eqref{usl0}.  One can identify a function with  it's exponents   $\alpha,...,\theta$.

\section{$3j$-symbols and Clebsh-Gordan coefficients}

\subsection{A relation to the Clebsh-Gordan coefficients}
\label{3jcg}

Let us be given representations  $V$, $W$, $U$ of the Lie algebra  $\mathfrak{gl}_3$. Choose in them bases  $\{v_{\mu}\}$, $\{w_{\nu}\}$, $\{u_{\rho}\}$. Then a  $3j$-symbol is a collection of numbers

\begin{equation}
\label{3j}
\begin{pmatrix}
V& W& U\\ v_{\mu} & w_{\nu} & u_{\rho}
\end{pmatrix}^s,
\end{equation}

 such that the value
 
  $$
  \sum_{\mu,\nu,\rho}\begin{pmatrix}
  V& W& U\\ v_{\mu} & w_{\nu} & u_{\rho}
  \end{pmatrix}^s    v_{\mu} \otimes  w_{\nu}  \otimes  u_{\rho}
  $$
  
   is  $\mathfrak{gl}_3$ semi-invariant.  The  $3j$-symbols with the same inner indices form a linear space.  An index  $s$ is indexing basic   $3j$-symbols with the same inner indices.
  
These coefficients are closely related to the Clebsh-Gordan coefficients. Indeed let us be given a decomposition into a sum of irreducible representations:
 \begin{equation}
 \label{rzl9}
   V\otimes W=\sum_s U^s.    \end{equation}
  
 Take basis $\{v_{\mu}\}$, $\{w_{\nu}\}$, $\{u^s_{\rho}\}$. The Clebsh-Gordan coefficients are coefficients in the decomposition
   
  
\begin{equation}
\label{ur}
%
%
u_{\rho'}^s = \sum_{\mu,\nu} D^{U,\rho',s}_{V,W;\mu,\nu} v_{\mu}\otimes w_{\nu}.
\end{equation}

Consider a  representation  $\bar{U}$, contragradient to  $U$, and take in   $\bar{U}$ a base  $\bar{u}_{\rho}$,  dual to   $u_{\rho}$
There exists a mapping  $U\otimes \bar{U} \rightarrow {\bf 1}$  into a trivial representation, such that  $u_{\rho'}\otimes \bar{u}_{\rho}\mapsto  \delta_{\rho,\rho'}$,  where  $\delta_{\rho,\rho'}$ is the Cronecer symbol.

Multiply  \eqref{ur} by  $\bar{u}_{\rho}$, take a sum over  $\rho$, one gets

$$
{\bf 1}=\sum_{\mu,\nu,\rho}D^{U,\rho',s}_{V,W;\mu,\nu} v_{\mu}\otimes w_{\nu}\otimes \bar{u}_{\rho}.
$$

Thus one has

$$
D^{U,\gamma,s}_{V,W;\mu,\nu}=\begin{pmatrix}
V& W& \bar{U}\\ v_{\mu} & w_{\nu} & \bar{u}_{\rho}
\end{pmatrix}^s.
$$
This formula allows to identify the multiplicity spaces for the  Clebsh-Gordan coefficients and for the  $3j$-symbols.

Thus the problem of calculation of the Clebsh-Gordan coefficients and  the  $3j$-symbols are equivalent.

\subsection{$3j$-symbols in functional realization}

In the functional realization a   $3j$-symbol for representations $V$, $W$, $U$  is described as follows.  Let the representations be realized in the space of functions on   $GL_3\times GL_3\times GL_3$.  We can put  $m_3=0$, $m'_3=0$.   Functions of matrix elements on the factors   $GL_3$  we denote as   $a_i^j$, $b_i^j$, $c_i^j$.  Analogous letters denote the determinants of matrices composed of matrix elements.

Decompose a  tensor product $V\otimes W\otimes U$ into a sum of irreducible representations and take  one of the occurring  trivial representations

  $$V\otimes W\otimes U={\bf 1}^s\oplus...,$$

 where ${\bf 1}^s$ is one of the occurring trivial representations.  More precise one can write  $V\otimes W\otimes U=\Big ( {\bf 1}\otimes M \Big ) \oplus...,$  where $M$ is a linear space called the multiplicity space. Choose in   $M$ a base  $\{e_i\}$ and denote ${\bf 1}^s:={\bf 1}\otimes e_s$.
In this representation several trivial representations can occur and   $s$ is their index  \footnote{Let us mention papers \cite{kl},  \cite{tk}, \cite{tk1}, where analogous approach to the decomposition of a double tensor product is used}.

Let base vectors be indexed by diagrams:

 \begin{align}
 \begin{split}
 \label{3d}
 &v_{\mu}=\begin{pmatrix}
 m_{1} && m_{2} &&0\\ &k_{1}&& k_{2}\\&&s
 \end{pmatrix},\,\,\,   w_{\nu}=\begin{pmatrix}
 m'_{1} && m'_{2} &&0\\ &k'_{1}&& k'_{2}\\&& s'
 \end{pmatrix},\\
 &u_{\rho}=\begin{pmatrix}
M_1 &&M_2 &&0\\ &K_1 &&K_2\\&&S
 \end{pmatrix}
 \end{split}
 \end{align}

 One has

 \begin{equation}
 \label{3ddu}
\bar{ u}_{\rho}=\begin{pmatrix}
 -M_3 &&-M_2 &&-M_1\\ &-K_2 &&-K_1\\&&-S
 \end{pmatrix}
 \end{equation}

 \subsection{  An explicit form of a  $\mathfrak{gl}_3$-invariant in  $V\otimes W\otimes U$}

  Let us prove that the highest vector of a representation  ${\bf 1}^s$  must be of type  (or linear combination of them)

  \begin{equation}
  \label{skob}
g=\prod_i (\underbrace{a\cdots a}_{k^i_1}\underbrace{b\cdots b}_{k^i_2}\underbrace{c\cdots c}_{k^i_3}),\end{equation}

  where, by analogy with \eqref{aab} as determinants we introduce expressions  $(abc)$, $(aac)$, $(acc)$, also put $(bcc)$, $(bbc)$. 
   \begin{align*}&(aabbcc):=(\tilde{a}\tilde{b}\tilde{c}),\,\,\,\tilde{a}_1^1:=a_{2,3},\,\,\,\tilde{a}_2^1:=-a_{1,3},\,\,\,\tilde{a}_3^1:=a_{1,2},\end{align*}
$\tilde{b}_i^1$,  $\tilde{c}_i^1$ are defined analogously.


 The following conditions must hold  
    
     \begin{align}
      \begin{split}
      \label{usl}
& m_1=\# \{ i : \,\,\, k_1^i=1     \}   ,\,\,\,\,  m_2=\# \{ i : \,\,\, k_1^i=2     \}   ,\,\,\,\,  0=\# \{ i : \,\,\, k_1^i=3     \},\\
& m'_1=\# \{ i : \,\,\, k_2^i=1     \}   ,\,\,\,\, m'_2=\# \{ i : \,\,\, k_2^i=2     \}   ,\,\,\,\,   0=\# \{ i : \,\,\, k_2^i=3     \},\\
& M_1=\# \{ i : \,\,\, k_1^i=1     \}   ,\,\,\,\, M_2=\# \{ i : \,\,\, k_1^i=2     \}  ,\,\,\,\,   M_3=\# \{ i : \,\,\, k_1^i=3     \}.\\
\end{split}
     \end{align}
 they provide that  $g\in V\otimes W\otimes U$ (см.  \cite{zh}). 
    

Indeed from one hand the index  $s$ indexing different   $3j$-symbols     \begin{equation}
\begin{pmatrix}
V& W& U\\ v_{\mu} & w_{\nu} & u_{\rho}
\end{pmatrix}^s
\end{equation}
    
 with the same inner indices coincides with the index numerating different Clebsh-Gordan coefficients
    
    $$ C^{\bar{U},\bar{\gamma},s}_{V,W;\alpha,\beta}.$$
    

From the other hand since to the Clebsh-Gordan coefficient to the index   $s$ there corresponds a function \eqref{foo}, to which there corresponds an expression of type \eqref{skob}, that is constructed as follows. To factors of  \eqref{foo} there correspond factors  form \eqref{skob} by the ruler:

  \begin{align*}
 &a_1 \mapsto (acc),\,\,\, a_{1,2}  \mapsto (aac),\,\,\,	b_1 \mapsto (bcc),\,\,\, b_{1,2}  \mapsto (bbc),\,\,\,	\\
 & (ab)\mapsto (abc),\,\,\, (aabb)_{1,2,1,3}  \mapsto (aabbcc),\,\,\,\\
 &(aab)\mapsto (aab),\,\,\, (abb)  \mapsto (abb).
 	\end{align*}
    

 Thus  starting from the expression  \eqref{foo}  we have constructed an expression \eqref{skob}, this construction is one-to-one. That is we have an isomorphism.   
    Hence we have descibed all  $\mathfrak{gl}_3$-invariants of a triple tensor product.
    
 
\section{A formula for a  $3j$-symbol}

Let us find an expression for a  $3j$-symbol. A $3j$-symbol has a multiplicity index  $s$, which is a function \eqref{foo}.  To this function there corresponds a trivial represnetation in the triple tensor product with the highest vector

 \begin{equation}
 \label{goo}
g(\omega,\varphi,\psi,\theta):=\frac{(acc)^{\alpha}(bcc)^{\beta}(aac)^{\gamma}(bbc)^{\delta}(abc)^{\omega}(abb)^{\varphi}(aab)^{\psi}(aabbcc)^{\theta}}{\alpha!\beta!\gamma!\omega!\varphi!\psi!\theta!}.
 \end{equation}

We have changed the expression \eqref{skob} by adding division by factorials of exponents.

\subsection{ Lattices $B'_1$ and  $B''_1$}

Consider independent variables corresponding to summands in determinants  $(caa),(acc),...,(aabbcc)$.  Denote these variables as follows 

\begin{align}
\begin{split}
\label{perem}
&Z=\{[[c_1a_{2,3}] ,[c_2a_{1,3}] ,[c_3a_{1,2}] , [a_1c_{2,3}] , [a_{2}c_{1,3}] ,[a_3c_{1,2}] ,
[c_1b_{2,3}]  ,  [c_2b_{1,3}] , [c_3b_{1,2}] ,\\& [b_1c_{2,3}] ,[b_{2}c_{1,3}] ,[b_3c_{1,2}] ,
[b_1a_{2,3}]  ,  [b_2a_{1,3}] , [b_3a_{1,2}], [a_1b_{2,3}] , [a_{2}b_{1,3}] ,[a_3b_{1,2}] ,\\&
[a_1b_2c_3],[a_2b_3c_1], [a_3b_1c_2], [a_2b_1c_3],[a_1b_3c_2], [a_3b_2c_1]
\\&
[a_{2,3}b_{1,3}c_{1,2}]  ,    [a_{1,3}b_{1,2}c_{2,3}]  ,    [a_{1,2}b_{2,3}c_{1,3}]  , 
[a_{1,3}b_{2,3}c_{1,2}]  ,   [a_{2,3}b_{1,2}c_{1,3}]  ,  [a_{1,2}b_{1,3}c_{2,3}] \}.
\end{split}
\end{align}

These variables are coordinates in a  $30$-dimentional space. When one opens brackets in determinants occuring in $g$,  there appear monomials in variables  \eqref{perem}. The vectors of exponents of these monomials are vectors in $30$-dimentional space.

Consider a vector 

$$
v_0=(\gamma,0,0,\alpha,0,0,\delta ,0,0,\beta,0,0,\psi,0,0,\varphi,0,0,\omega,0,0,0,0,0,\theta,0,0,0,0,0)
$$

Such a vector of exponents is obtained in the case when one takes the first summand in each determinant. 
Note that either  $\omega=0$,  or   $\theta=0$.

If one changes a choice of summands that to the vector of exponents  $v_0$ one need to add one of the following vectors 

\begin{align*}
&p_1= e_{[c_1a_{2,3}]}-e_{ [c_2a_{1,3}] }, &   p_2=e_{[c_1a_{2,3}]}-e_{ [c_3a_{1,2}] }, \\&p_3= e_{[a_1c_{2,3}] }-e_{[a_2c_{1,3}] }, & p_4= e_{[a_1c_{2,3}] }-e_{[a_3c_{1,2}] },\\
&p_5= e_{[c_1b_{2,3}]}-e_{ [c_2b_{1,3}] }, &   p_6=e_{[c_1b_{2,3}]}-e_{ [c_3b_{1,2}] }, \\& p_7= e_{[b_1c_{2,3}] }-e_{[b_2c_{1,3}] }, & p_8= e_{[b_1c_{2,3}] }-e_{[b_3c_{1,2}] },\\
&p_9= e_{[a_1b_{2,3}]}-e_{ [a_2b_{1,3}] }, &   p_{10}=e_{[a_1b_{2,3}]}-e_{ [a_3b_{1,2}] }, \\& p_{11}= e_{[b_1a_{2,3}] }-e_{[b_2a_{1,3}] }, & p_{12}= e_{[b_1a_{2,3}] }-e_{[b_3a_{1,2}] },\\&
p_{13}=e_{[a_1b_2c_3]}-e_{[a_2b_3c_1]} & p_{14}=e_{[a_1b_2c_3]}-e_{[a_3b_1c_2]} \\
& p_{15}=e_{[a_1b_2c_3]}-e_{[a_2b_1c_3]} & p_{16}=e_{[a_1b_2c_3]}-e_{[a_1b_3c_2]}  \\& p_{17}=e_{[a_1b_2c_3]}-e_{[a_3b_2c_1]},
&p_{18}= e_{[a_{2,3}b_{1,3}c_{1,2}]}-e_{ [a_{1,3}b_{1,2}c_{2,3}] }, \\&   p_{19}= e_{[a_{2,3}b_{1,3}c_{1,2}]}-e_{ [a_{1,2}b_{2,3}c_{1,3}] }, &   p_{20}= e_{[a_{2,3}b_{1,3}c_{1,2}]}-e_{ [a_{1,3}b_{2,3}c_{1,2}] },  \\&   p_{21}= e_{[a_{2,3}b_{1,3}c_{1,2}]}-e_{ [a_{1,2}b_{1,3}c_{2,3}] }, &   p_{22}= e_{[a_{2,3}b_{1,3}c_{1,2}]}-e_{ [a_{2,3}b_{1,2}c_{1,3}] },
\end{align*}

Define projectors  

$$
pr_a,pr_b,pr_c:\mathbb{C}^{30}\rightarrow \mathbb{C}^6,
$$

which operate as follows. Given a vector of exponents for variables  \eqref{perem} they construct vectors of exponents for variables $a_X$, $b_X$, $c_X$.

Conside a vector  $v$ of exponents of a monomial in variables  \eqref{perem}, which corresponds to an arbitrary choice of summands in determinants. Let us find which vectors   $\tau\in\mathbb{Z}<p_1,...,p_{22}>$ has the following property. When one adds it to $v$ then to the projections  $pr_a$, $pr_b$, $pr_c$ vectors proportional to  $(0,1,-1,-1,1,0)$ are added.

\begin{defn} {\it An elementary cycle} is the following object.  Take two determinant from  $(caa),...,(aabbcc)$  and draw two arrows from a symbol  $x=a,b,c$  in one determinant to a symbol   $xx=aa,bb,cc$ in another determinant. Or from a symbol $x$ in one determinant to   $x$ in another determinant. Or from a symbol   $xx$ in one determinant to a symbol   $xx$ in another determinant. 	

The following two condition must be satisfied. One arrow goes from the first determinant to the second and another goes form the second to the first. And one of the arrows goes from   $x$ to $xx$.
\end{defn}

Here are examples of cycles:

\newpage

\[
(a\tikzmark{0}  bb\tikzmark{11} ) (\tikzmark{00} aa \tikzmark{1} b),\,\,\, 
(a\tikzmark{5}  b\tikzmark{66} c) (\tikzmark{55} aa \tikzmark{6} bb cc),\,\,\,(a\tikzmark{90}  bb\tikzmark{911} ) (\tikzmark{900} a \tikzmark{91} bc),\,\,\,
(a\tikzmark{95}  bb\tikzmark{966} ) (\tikzmark{955} aa \tikzmark{96} bb cc).
\]

\begin{tikzpicture}[remember picture, overlay, bend left=45, -latex, blue]
\draw ([yshift=2ex]pic cs:0) to ([yshift=2ex]pic cs:00);
\draw ([yshift=0ex]pic cs:1) to ([yshift=0ex]pic cs:11);
\draw ([yshift=2ex]pic cs:5) to ([yshift=2ex]pic cs:55);
\draw ([yshift=0ex]pic cs:6) to ([yshift=0ex]pic cs:66);
\draw ([yshift=2ex]pic cs:90) to ([yshift=2ex]pic cs:900);
\draw ([yshift=0ex]pic cs:91) to ([yshift=0ex]pic cs:911);
\draw ([yshift=2ex]pic cs:95) to ([yshift=2ex]pic cs:955);
\draw ([yshift=0ex]pic cs:96) to ([yshift=0ex]pic cs:966);
\end{tikzpicture}

All other examples of cycles are obtained form these by applying a permutation of symbols  $a,b,c$.

To an elementary cycle there corresponds a vector  $\tau\in\mathbb{Z}<p_1,...,p_{22}>$ by the following ruler. To the elementary cycles written above there correspond vectors   

\begin{align}
\begin{split}
\label{vec}
& u_1=-e_{[a_1b_{2,3}]}+e_{[a_2b_{1,3}]}-e_{[b_1a_{2,3}]}+e_{[b_2a_{1,3}]},\\
& u_2=-e_{[a_1b_{2}c_{3}]}+e_{[a_2b_{1}c_{3}]}-e_{[a_{2,3}b_{1,3}c_{1,2}]}+e_{[a_{1,3}b_{2,3}c_{1,2}]},\\
& u_3=-e_{[a_1b_{2,3}]}+e_{[a_2b_{1,3}]}-e_{[a_{2,3}b_{1,3}c_{1,2}]}+e_{[a_{1,3}b_{2,3}c_{1,2}]},\\
& u_4=-e_{[a_1b_{2,3}]}+e_{[a_2b_{1,3}]}-e_{[a_{2,3}b_{1,3}c_{1,2}]}+e_{[a_{1,3}b_{2,3}c_{1,2}]},\\
\end{split}
\end{align}

Let  $B_1$  be an integer lattice spanned by the vectors correponding to elementary cycles. One has

\begin{prop}
	\label{prp}
	
	The lattice  $B_1$ is generated by vectors correspondig to elementary cycles of type

	\[
	(a\tikzmark{995}  bb\tikzmark{9966} ) (\tikzmark{9955} aa \tikzmark{996} bb cc),\,\,\,
	(a\tikzmark{85}  bb\tikzmark{866} ) (\tikzmark{855} a \tikzmark{86} b c) 
	\]

	\begin{tikzpicture}[remember picture, overlay, bend left=45, -latex, blue]
	\draw ([yshift=2ex]pic cs:995) to ([yshift=2ex]pic cs:9955);
	\draw ([yshift=0ex]pic cs:996) to ([yshift=0ex]pic cs:9966);
	\draw ([yshift=2ex]pic cs:85) to ([yshift=2ex]pic cs:855);
	\draw ([yshift=0ex]pic cs:86) to ([yshift=0ex]pic cs:866);
	\end{tikzpicture}
\end{prop}

\proof

One proves directely that a vector corresponding to an elementary elementary cycle can be expressed through these vectors. 

\endproof

One has $B_1\subset B$.

In Proposition \ref{prp} two series of generators of the lattice   $B_1$ are listed.	 Let  $B'_1$ be a sublattice of   $B_1$, generated by elementary cycles of the first type. Let  $B''_1$ be a sublattice of   $B_1$, generated by elementary cycles of the second type.

In both cases the generators are  basis. The fact that the selected vectors are linearly independend follows form the fact that each vector containes coordinate vectors that are not involved in other vectors.

Thus in  $B'_1$ there exists a base

\begin{align}
\begin{split}
\label{bb1}
&v_1^a=-e_{[a_1b_{2,3}]}+e_{[a_2b_{1,3}]}-e_{[a_{2,3}b_{1,3}c_{1,2}]}+e_{[a_{1,3}b_{2,3}c_{1,2}]},\,\,\, v_2^a=-e_{[a_1c_{2,3}]}+e_{[a_2c_{1,3}]}-e_{[a_{2,3}b_{1,2}c_{1,3}]}+e_{[a_{1,3}b_{1,2}c_{2,3}]}\\
&v_1^b=-e_{[b_1a_{2,3}]}+e_{[b_2a_{1,3}]}-e_{[a_{1,3}b_{2,3}c_{1,2}]}+e_{[a_{2,3}b_{1,3}c_{1,2}]},\,\,\, v_2^b=-e_{[b_1c_{2,3}]}+e_{[b_2c_{1,3}]}-e_{[a_{1,2}b_{2,3}c_{1,3}]}+e_{[a_{1,2}b_{1,3}c_{2,3}]}\\
&v_1^c=-e_{[c_1a_{2,3}]}+e_{[c_2a_{1,3}]}-e_{[a_{1,3}b_{1,2}c_{2,3}]}+e_{[a_{2,3}b_{1,2}c_{1,3}]},\,\,\, v_2^c =-e_{[c_1b_{2,3}]}+e_{[c_2b_{1,3}]}-e_{[a_{1,2}b_{1,3}c_{2,3}]}+e_{[a_{1,2}b_{2,3}c_{1,3}]}
\end{split}
\end{align}

In the case when \eqref{goo}  contains   $(aabbcc)$ but does not contain  $(abc)$ one uses $B'_1$. In the case when \eqref{goo} contains $(abc)$ but does not contain  $(aabbcc)$ one uses  $B''_1$.

\begin{prop}
	
 Let 	 \eqref{goo} contain  $(aabbcc)$  and does not contain $(abc)$.  Take $30$-dimensional vectors  $\varpi$ ,  $\varpi'$ of exponents of two monomials in variables  $Z$ which one obtains when one opens brackets in  \eqref{goo}.  Let $[\mu,\nu,\rho]=[pr_a(\varpi),pr_b(\varpi),pr_c(\varpi)]$,    $[\mu',\nu',\rho']=[pr_a(\varpi'),pr_b(\varpi'),pr_c(\varpi')]$.

	 Then  $[\mu',\nu',\rho'] =[\mu,\nu,\rho]+b_1$, $b_1\in B'_1$ if and only if  simultaneously
	
	\begin{align*}& \mu=\mu' mod (0,-1,1,1,-1,0),
	& \nu=\nu' mod  (0,-1,1,1,-1,0),
	& \rho=\rho' mod  (0,-1,1,1,-1,0).\end{align*}
	
 In the case when \eqref{goo} contains  $(abc)$ but does not contain   $(aabbcc)$ the same is true  if one changes  $B'_1$  to  $B''_1$.
	
\end{prop}

\begin{proof}
	Consider the case when   \eqref{goo}  contains  $(aabbcc)$  but does not contain  $(abc)$.

From  $[\mu',\nu',\rho'] =[\mu,\nu,\rho] +b_1$, $b_1\in B'_1$  it follows that
	$\mu=\mu' mod (0,-1,1,1,-1,0)$,
	$  \nu=\nu' mod  (0,-1,1,1,-1,0)$,
	$\rho=\rho' mod  (0,-1,1,1,-1,0)$. It can be proved by direct computations. 
	
	Let us prove the contraverse.
One changes the choice of summands in the determinants, such that  	$\mu=\mu' mod (0,-1,1,1,-1,0)$,
$  \nu=\nu' mod  (0,-1,1,1,-1,0)$,
$\rho=\rho' mod  (0,-1,1,1,-1,0)$.  For example we change  $a_1b_{2,3}$  from the determinant    $(abb)$  to  $a_2b_{1,3}$.   The  exponent of   $b_{2,3}$ reduces by  $1$ and the exponent of   $b_{1,3}$ increases by  $1$. The vector of exponents of determinants  $b$ must belong to the set  $\delta+ (0,-1,1,1,-1,0)$. Thus we the exponent of   $b_1$ must increase by  $1$ and the exponent of  $b_2$  must decrease by  $1$.  This can take place in the case when in another determinant say in   $(bcc)$ the summand $b_1c_{2,3}$ is changed to  $b_2c_{1,3}$. Then we consider the cahange of exponents of   $c_{2,3}$, $c_{1,3}$ and so on.  Finally we must return to the determinant  $(abb)$ and obtain the  change of exponent of determinants $a_1$ and   $a_2$ which takes place due to the change of a summand in the first determinant.
	
 This construction corresponds to the shift of the vector of exponents by a vector from $B'_1$.
	
\end{proof}

\subsubsection{Scalar  products}

Consider the case when   \eqref{goo} contains  $(aabbcc)$  and does not contain  $(abc)$.

Let us calculate scalar products. For  $s_1,s_2,s_3\in\mathbb{Z}_{\geq 0 }$ introduce a hypergeometric type series in variabels  $Z$:

\begin{align}
\begin{split}
\label{defz}
&\mathcal{F}_{\varpi}^{s_1,s_2,s_3}(Z,B'_1):=\sum_{t\in \mathbb{Z}^6}   \binom{t_1^a+t_2^a+s_1}{s_1}
\binom{t_1^b+t_2^b+s_2}{s_2}
\binom{t_1^c+t_2^c+s_3}{s_3}
\frac{Z^{\varpi+tv}}{  (\varpi+tv)!   },\\
&t=(t_1^1,t_2^a,t_1^b,t_2^b,t_1^c,t_2^c),\,\,\,\,tv=t^a_1v^a_1+t^a_2v^a_2+...,\,\,\, \varphi\in\mathbb{Z}^{30}\\
\end{split}
\end{align}

Introduce vectorw

\begin{align*}
&	f_a=-e_{[a_1c_{2,3}]}+e_{[a_3c_{1,2}]}-e_{[a_{2,3}b_{1,3}c_{1,2}]}-e_{  [a_{1,2}b_{1,3}c_{2,3} ] },\\
&f_b=-e_{[b_1c_{2,3}]}+e_{[b_3c_{1,2}]}-e_{[a_{1,3}b_{2,3}c_{1,2}]}-e_{  [a_{1,3}b_{1,2}c_{2,3} ] },\\
&f_c=-e_{[a_1b_{2,3}]}+e_{[a_3b_{1,2}]}-e_{[a_{1,2}b_{2,3}c_{1,3}]}-e_{  [a_{2,3}b_{1,2}c_{1,3} ] }.
\end{align*}

One has  (see. \eqref{vr})

\begin{align}
\begin{split}
\label{prabc}
&pr_a(\tau+f_a)=pr(\tau)+r,\\
&pr_b(\tau+f_a)=pr(\tau),\\
&pr_c(\tau+f_a)=pr(\tau),
\end{split}
\end{align}

and one has analogous conditions for  $f_b,f_c$.

Below we consider vectors $\varphi\in\mathbb{Z}^{30}$ and vectors $\mu,\nu,\rho$, that satisfy the following relations 

\begin{align}
\begin{split}
\label{vrph}
&\varpi=v_0+tv^{abc}+s_1f_a+s_2f_b+s_3f_c,\,\,\, t\in\mathbb{Z}^{9},\,\,\, s_1,s_2,s_3\in\mathbb{Z},\\
&\mu=pr_a(\varphi),\,\,\, \nu=pr_b(\varpi),\,\,\, \rho=pr_c(\varpi).
\end{split}
\end{align}

\begin{prop}
	\label{prp1}
Let  $\varpi, \mu,\nu,\rho$ satisfy the relations  \eqref{vrph}. 	 Then
	
	\begin{equation}
	\label{p1}
	<g,\mathcal{F}_{\mu}^{s_1}(A)\mathcal{F}_{\nu}^{s_2}(B)\mathcal{F}_{\rho}^{s_3}(C)>=
	\mathcal{F}_{\varpi}^{s_1,s_2,s_3}(\pm 1,B'_1),
	\end{equation}
	

Instead od those variables from $Z$ (see \eqref{perem}) which occur in a determinant with a sign $+$,  one substituts  $+1$, and instead of those  variables  which occur in a determinant with a sign  $-$, one substituts  $-1$.
	
\end{prop}

\proof

Consider a product   $\frac{A^x}{x!}\frac{B^y}{y!}\frac{C^z}{z!}$, occuring in $\mathcal{F}_{\mu}^{s_1}(A)\mathcal{F}_{\nu}^{s_2}(B)\mathcal{F}_{\rho}^{s_3}(C)$.  A coefficient at this product is defined as follows. If 
$x=\mu+\tau_1v$, $y=\nu+\tau_2 v$, $z=\rho+\tau_3v$, then the coefficient equals

\begin{equation}
\label{kf1}
\binom{\tau_1+s_1}{s_1}
\binom{\tau_2+s_2}{s_2}
\binom{\tau_3+s_3}{s_3}
\end{equation}

Now consider a scalar product

\begin{equation}
\label{gabc}
<g,\frac{A^x}{x!}\frac{B^y}{y!}\frac{C^z}{z!}>.
\end{equation}

This scalar product can be calculated as follows. In   \eqref{goo} the brackets are opened.  One obtains an expression which is a    $\Gamma$-series in variables  $Z$, where the variables occuring in a determinant with a  sign $-$ are taken with a sign  $-$.  
After this the variables $Z$  are replacet to products of variables  $A_X,B_X,C_X$ in an obvious way. 

Let   $h$  be a vector of exponents of a monomila in variables  $Z$,  which appears when one opens brackets in  \eqref{goo}.  Sucj a monomial is divided by factorials of its powers and is multiplied by  sign   $\pm$.
 The scalar  product \eqref{gabc} is non-zero if for some   $h$ one has  $[pr_a(h),pr_b(h),pr_c(h)]=[x,y,z]$. In this case  \eqref{gabc}  is obtained by a change in a suitable monomial of all variables  $Z$  to $1$ and a summation over all suitable monomials.

Finally let us write  $x=\mu+\tau_1v$, $y=\nu+\tau_2v$, $z=\rho+\tau_3v$. In the case $\tau_1=\tau_2=\tau_3=0$ one has  $[\mu,\nu,\rho]=[pr_a(\varpi),pr_b(\varpi),pr_c(\varpi)]$. Thus
$[x,y,z]=[pr_a(\varpi'),pr_b(\varpi'),pr_c(\varpi')]$, where 

\begin{align*}
& \varphi'=\varpi+t^a_1 v^a_1+t^a_2v^a_2+
t^b_1 v^b_1+t^b_2v^b_2+
t^c_1 v^c_1+t^c_2v^c_2,\\
&\tau_1=t^a_1+t^a_2,\,\,\, \tau_2=t^b_1+t^b_2,\,\,\, \tau_3=t^c_1+t^c_2,\,\,\, \end{align*}

Let us take into consideration the coefficient  \eqref{kf1} at  $\frac{A^x}{x!}\frac{B^y}{y!}\frac{C^z}{z!}$ ,occuring in  $\mathcal{F}_{\mu}^{s_1}(A)\mathcal{F}_{\nu}^{s_2}(B)\mathcal{F}_{\rho}^{s_3}(C)$.   As the result one gets an expression in the right hand side in  \eqref{p1}.

\endproof

Introduce vectors

\begin{align*}
&	f_a=-e_{[a_1c_{2,3}]}+e_{[a_3c_{1,2}]}-e_{[a_{2,3}b_{1,3}c_{1,2}]}-e_{  [a_{1,2}b_{1,3}c_{2,3} ] },\\
&f_b=-e_{[b_1c_{2,3}]}+e_{[b_3c_{1,2}]}-e_{[a_{1,3}b_{2,3}c_{1,2}]}-e_{  [a_{1,3}b_{1,2}c_{2,3} ] },\\
&f_c=-e_{[a_1b_{2,3}]}+e_{[a_3b_{1,2}]}-e_{[a_{1,2}b_{2,3}c_{1,3}]}-e_{  [a_{2,3}b_{1,2}c_{1,3} ] }.
\end{align*}

One has relations 

\begin{align}
\begin{split}
\label{prabc}
&pr_a(\tau+f_a)=pr(\tau)+r,\\
&pr_b(\tau+f_a)=pr(\tau),\\
&pr_c(\tau+f_a)=pr(\tau),
\end{split}
\end{align}

and analogous relations for  $f_b,f_c$.

\begin{prop}
	\label{prp2}
	Let  $\varpi, \mu,\nu,\rho$ satisfy the relations  \eqref{vrph}. 	 Then
	$$
	<g,\mathcal{F}_{\mu-s_1r}^{s_1}(A)\mathcal{F}_{\nu-s_2r}^{s_2}(B)\mathcal{F}_{\rho-s_3r}^{s_3}(C)>=
	\mathcal{F}_{\varpi-s_1f_a-s_2f_b-s_3f_c}^{s_1,s_2,s_3}(\pm 1,B_1),
	$$
Instead of those variables from $Z$ (see \eqref{perem}) which occur in a determinant with a sign $+$,  one substituts  $+1$, and instead of those  variables  which occur in a determinant with a sign  $-$, one substituts  $-1$.
	
\end{prop}

\proof

This Proposition  follows imidiately from Proposition  \ref{prp1} and the formula \eqref{prabc}.

\endproof

Introduce a function

\begin{equation}
\label{defz1}
\tilde{F}_{\varpi}(Z,B'_1):=\sum_{s_1,s_2,s_3\in\mathbb{Z}_{\geq 0}}(-1)^{s_1+s_2+s_3}	\mathcal{F}_{\varpi-s_1f_a-s_2f_b-s_3f_c}^{s_1,s_2,s_3}(Z,B'_1).
\end{equation}

Note that the function  $\mathcal{F}_{\varpi}^{0,0,0}(Z,B'_1)$ satisfies the GKZ equation and the function   $\tilde{F}_{\kappa}(Z,B'_1)$  satisfies the A-GKZ equaiton which consists of  3 equations of type

$$
(\frac{\partial^2}{\partial [a_1c_{2,3}]\partial [a_{2,3}b_{1,2}c_{1,3}]}-\frac{\partial^2}{\partial [a_2c_{1,3}]\partial [a_{1,3}b_{1,2}c_{2,3}]}+\frac{\partial^2}{\partial [c_3a_{1,2}]\partial [a_{1,2}b_{1,3}c_{2,3}]})\tilde{F}_{\varpi}(Z,B'_1)=0.
$$


As a direct consequence of Proposition  \ref{prp2}  one gets

\begin{prop}
	Let  $\varpi, \mu,\nu,\rho$ satisfy the relations  \eqref{vrph}. 	 Then
	$$<g,\tilde{F}_{\mu}(A)\tilde{F}_{\nu}(B)\tilde{F}_{\rho}(C)>=\tilde{F}_{\kappa}(\pm 1,B'_1).$$
	
Instead of those variables from $Z$ (see \eqref{perem}) which occur in a determinant with a sign $+$,  one substituts  $+1$, and instead of those  variables  which occur in a determinant with a sign  $-$, one substituts  $-1$.
	
\end{prop}

Introduce a function

\begin{equation}
\label{funs}
F_{\varpi}(Z,B'_1):=\sum_{s_1,s_2,s_3\in\mathbb{Z}_{\geq 0}}f^{pr_a(\varpi)}_{s_1}f^{pr_b(\varpi)}_{s_2}f^{pr_c(\varpi)}_{s_3}F_{\varpi-s_1f^a-s_2f^b-s_3f^c}(Z,B'_1),
\end{equation}

where coefficients $f$  are defined in  \eqref{kfs}. Using the relation   \eqref{fs0}, one gets the following statement

\begin{prop}
Let   \eqref{goo} contain   $(aabbcc)$ and does not contain  $(abc)$, let  $\varpi, \mu,\nu,\rho$ satisfy the relations  \eqref{vrph}. 	 Then

	$$<g,F_{\mu}(A)F_{\nu}(B)F_{\rho}(C)>=F_{\varpi}(\pm 1,B'_1)$$
	
	Let   \eqref{goo} contain  $(abc)$ and does not contain  $(aabbcc)$, then

	$$<g,F_{\mu}(A)F_{\nu}(B)F_{\rho}(C)>=F_{\varpi}(\pm 1,B''_1)$$
	
Instead of those variables from $Z$ (see \eqref{perem}) which occur in a determinant with a sign $+$,  one substituts  $+1$, and instead of those  variables  which occur in a determinant with a sign  $-$, one substituts  $-1$.
\end{prop}

\subsection{ $3j$-symbols}

\subsection{Selection rulers for   $3j$-symbols}
Let us answer the following question: which products $\mathcal{F}_{\mu}(a)\mathcal{F}_{\nu}(b)\mathcal{F}_{\rho}(c)$  can have a non-zero scalar product with a function   $g$,  of type  \eqref{goo}.  

Let us consider an A-GKZ realization.  In this realization the vector $\mathcal{F}_{\mu}(a)\mathcal{F}_{\nu}(b)\mathcal{F}_{\rho}(c)$   is presented by an expression of type    $F_{\mu}(A)F_{\nu}(A)F_{\rho}(C)$.  The function  $g$  is presented by an expression of type  $g+pl$, where  $pl$ is proportional to  $A_1A_{2,3}-A_{2}A_{1,3}+A_{3}A_{1,2}$,.... Since the function  $F_{\mu}(A),F_{\nu}(A),F_{\rho}(C)$ is a solution of the A-GKZ system, one has

$$
<F_{\mu}(A)F_{\nu}(A)F_{\rho}(C), g+pl>=<F_{\mu}(A)F_{\nu}(A)F_{\rho}(C), g>.
$$

Thus we need to consider the scalar product    $<F_{\mu}(A)F_{\nu}(A)F_{\rho}(C), g>$.

A support of a function presented as a power series is a set of vectors of exponents of it's monomials.
Note that the support of a function $g$, defined in  \eqref{goo}, considered as a function of determinants
$a_X,b_X,c_X$
is the following

$$
supp(g)=(\kappa+B)\cap (\mathbb{Z}_{\geq 0}^{18})
$$
for some vector  $\kappa\in\mathbb{Z}_{\geq 0}^{18}$ and some lattice  $B$.  The vector  $\kappa$ and the generators of  $B$ can be written in the following manner:

\begin{align*}
&\kappa=[pr_a(v_0),pr_b(v_0),pr_c(v_0)],
& B=\mathbb{Z}<\pi_1,...,\pi_{22}>,\,\,\,\,\, \pi_i=[pr_a(p_i),pr_b(p_i),pr_c(p_i)]
\end{align*}

The function $F_{\mu}(A)F_{\nu}(A)F_{\rho}(C)$ can have a non-zero scalar product with  $g$  only if the following intersection is non-empty

$$
supp(F_{\mu}(A)F_{\nu}(A)F_{\rho}(C)) \cap supp(g)
$$

This condition is written explicitely as follows

\begin{align}
\begin{split}
\label{potb}
&(\bigcup_{s_1,s_2,s_3\in \mathbb{Z}_{\geq 0}} [mu+B_{GC}-s_1r, nu+B_{GC}-s_1r,rho+B_{GC}-s_2r ]\cap (\mathbb{Z}_{\geq 0}^{18})) \cap supp(g)\neq \emptyset
\end{split}
\end{align}

A condition for a vector  $\varphi$, than provides us that the vector  $[\mu,\nu,\rho]=[pr_a(\varpi),pr_b(\varpi),pr_c(\varpi)]$  satisfies  \eqref{potb}  is just the condition  \eqref{vrph}.

\subsection{A formula for a $3j$-symbol}

Let us write
\begin{equation}
g=\sum_{\mu\nu\rho} c_{\mu',\nu',\rho'} \mathcal{F}_{\mu'}(a)\mathcal{F}_{\nu'}( b)\mathcal{F}_{\rho'}(c),
\end{equation}
where $ c_{\mu',\nu',\rho'}$  is a  $3j$-symbol  \eqref{3j}.  Change  determinants $a_X$, $b_X$,  $c_X$ to independent variables $A_X$, $B_X$, $C_X$. Then the equalities do hold only the Plucker relations.

Take a scalar product of both sides of this equality with $F_{\mu}(A)\mathcal{F}_{\nu}(B)\mathcal{F}_{\rho}(C)$.  Since these functions are solution of the A-GKZ system one can ignore the fact that the previous equality holds only the Plucker relation.

Using the previous caculations one obtaines the Theorem.


\begin{thm}
	\label{ost}
	Let us be given representations with the highest weights  $[m_1,m_2,0]$, $[m'_1,m'_2,0]$, $[M_1,M_2,M_3]$. Let be given Gelfand-Tsetlin base vectors, the    $\Gamma$-series that correspond to them have shift vectors   $\mu,\nu,\rho$ (the formula for the shift vectors see in Theorem \ref{vec3}). Fix a fucntion of type  \eqref{goo} with exponents satisfying conditions of Proposition \ref{mlt}. 
	
	 Take a vector  $\varpi$ wich is related by the equalities   \eqref{vrph}    with the vectors   $\mu,\nu,\rho$ (if it is not possible, than the  $3j$-symbol equals to zero).

	Then a    $3j$-symbol  \eqref{3j}     equals
	\begin{align}
	\begin{split}
	\label{osntf2}
	&\frac{F_{\varpi}(\pm 1,B'_1)}{\mathcal{F}_{\mu}(1)\mathcal{F}_{\nu}(1)\mathcal{F}_{\rho}(1)}, \text{ if  \eqref{goo} does not contian  $(abc)$},\\
	&\frac{F_{\varpi}(\pm 1,B''_1)}{\mathcal{F}_{\mu}(1)\mathcal{F}_{\nu}(1)\mathcal{F}_{\rho}(1)}, \text{ if \eqref{goo} does not contain  $(aabbcc)$},
	\end{split}
	\end{align}   where the function occuring in the numerator is defined in   \eqref{funs} (see also  \eqref{defz1}, \eqref{defz}).

	In the function  $F_{\varpi}$ in the numerator we substitute $+1$ instead of those   $Z$ (see \eqref{perem}), that occur in determinants in   \eqref{goo} with the sign  $+$, and instead of    those   $Z$ (see \eqref{perem}), that occur in determinants in   \eqref{goo} with the sign  $-$  we substitute $-1$.
	
	In the denominator the $\Gamma$-series introduced in Theorem \ref{vec3} occur. WE substitute instead of all their arguments   $1$.

\end{thm}

\section{ Appendix: simpler selection rulers}
\label{dop}

In this section we present conditons under which a Clebsh-Gordan coefficient or a   $3j$-symbol can be non-zero.

\subsection{Selection rulers for Clebsh-Gordan coefficients}

Let us find necessary condition for the elements of the diagrams \eqref{3d}  under which the Clebsh-Gordan coefficient  $C^{U,\gamma,s}_{V,W;\alpha,\beta}$ can be non-zero.

Let us find condition for upper rows that are  $\mathfrak{gl}_3$-highest weights. Let the index  $s$  correspond to the function \eqref{foo}.  In tern of exponents of  \eqref{foo} these upper rows are written as follows 

\begin{align}
\begin{split}
\label{s3d}
&[m_1,m_2,0]=[\alpha+\gamma+\omega+\varphi+\psi+\theta,\gamma+\psi+\theta,0],\\
&[m'_1,m'_2,0]=[\beta+\delta+\omega+\varphi+\psi+\theta, \delta+\varphi+\theta ,0]\\
&[M_1,M_2,M_3]=[\alpha+\beta+\gamma+\delta+\omega+\varphi+\psi+2\theta, \gamma+\delta+\omega+\varphi+\psi+\theta,\varphi+\psi+\theta]\\
\end{split}
\end{align}

One sees that between these rows one has a relation
\begin{align}
\begin{split}
\label{s3}
&[m_1,m_2,0]+[m'_1,m'_2,0]+\omega[-1,1,0]+(\varphi+\psi)[-1,0,1]+\theta[0,-1,1]=\\&=[M_1,M_2,M_3].
\end{split}
\end{align}

Moreover this relation is sufficient for existence of  $\alpha,...,\omega$, for which the conditions \eqref{s3d} hold.

Let us find selection rulers for the second rows of diagrams \eqref{3d}.  One can obtain them by considering the selection rulers for the Clebsh-Gordan coefficients for the algebra $\mathfrak{gl}_2$. One has a tensor product of representations of the algebra  $\mathfrak{gl}_2$ with highest weights   $[k_1,k_2]$  and  $[\bar{k}_1,\bar{k}_2]$.  Which highest weights  $[K_1,K_2]$ occur in a decomposition of this tensor product?

A base in the space of  $\mathfrak{gl}_2$-highest vectors in the functional realization is foremed by functions of type 

\begin{equation}
\label{foo2}
f=a_1^{\alpha}a_{1,2}^{\beta}b_1^{\gamma}b_{1,2}^{\delta}(ab)^{\omega}.
\end{equation}

Then

\begin{align}
\begin{split}
\label{2d}
&[k_1,k_2]=[\alpha+\beta+\omega,\beta],\\
&[k'_1,k'_2]=[\gamma+\delta+\omega,\delta],\\
&[K_1,K_2]=[\alpha+\beta+\gamma+\delta+\omega,\beta+\delta+\omega].
\end{split}
\end{align}

One sees that the following condition holds
\begin{align}
\begin{split}
\label{s2}
[k_1,k_2]+[k'_1,k'_2]+\omega[-1,1]=[K_1,K_2].
\end{split}
\end{align}

The relations  \eqref{s2} are sufficient for existence of  $\alpha,...,\omega$, for which
the conditions \eqref{2d} hold.

Considering  $E_{1,1}$-weights of diagrams one concludes that
for the third rows one has

\begin{equation}
\label{s1}
s+s'=S.
\end{equation}

Thus we have proved the following.

\begin{prop}
	A Clebsh-Gordan coefficient can be non-zero only if conditions   \eqref{s3}, \eqref{s2}, \eqref{s3} hold for some non-negative integers  $\omega$, $\varphi$, $\psi$, $\theta$.
\end{prop}



\end{document}